\documentclass{elsarticle}

\usepackage{amsmath, amscd} 
\usepackage{amsthm} 
\usepackage{amssymb}
\usepackage{amsfonts} 
\usepackage{qsymbols} 
\usepackage{latexsym}

\newcommand {\R}{{\mathbb R}}

\newcommand {\N}{{{\mathbb N}}}

\newcommand {\C}{{\mathbb C}}

\newcommand {\F}{{\mathcal{F}}}
\newcommand {\Ri}{{\mathrm{R}}} 
\newcommand {\w}{{\omega}} 
\newcommand{\al}{{\alpha}} 
\newcommand {\ph}{{\varphi}} 
 
\newcommand {\AM}{{\mathrm{AM}}} 
\newcommand{\Ma}{{\mathcal{M}}} 
\newcommand {\HT}{{\mathrm{H}}} 
 
\newcommand {\La}{{\mathcal{L}}} 
\newcommand{\supp}{{\mathrm{supp}}} 
\newcommand{\abs}[1]{\lvert#1\rvert}
\newcommand {\ran}{{\mathrm{ran}}}
\newcommand{\eM}{\mathrm{M}} 
\newcommand{\dom}{\mathrm{dom}}
\newcommand{\ud}{\mathrm{d}} 
\newcommand{\norm}[1]{\left\|#1\right\|}
\newcommand{\Ell}{\mathrm{L}} 
\newcommand{\Ce}{\mathrm{C}}
\newcommand{\vanish}[1]{\relax}

\DeclareMathOperator{\Real}{Re}

\newtheorem{theorem}{Theorem}[section]
\newtheorem{lemma}[theorem]{Lemma}
\newtheorem{proposition}[theorem]{Proposition}
\newtheorem{corollary}[theorem]{Corollary}

\theoremstyle{definition} 

\newtheorem{remark}[theorem]{Remark}

\numberwithin{equation}{section}

\begin{document}

\begin{frontmatter}

\title{Functional Calculus for Semigroup Generators via Transference}

\author{Markus Haase} 
\ead{m.h.a.haase@tudelft.nl}

\author{Jan Rozendaal\corref{cor}\fnref{NWO}} 
\ead{J.Rozendaal-1@tudelft.nl} 

\address{Delft Institute of Applied
Mathematics\\Mekelweg 4\\2628CD Delft\\The Netherlands}

\fntext[NWO]{Supported by NWO-grant
613.000.908 ``Applications of Transference Principles".}

\cortext[cor]{Corresponding author}

\date{July 12, 2013}

\begin{abstract} 
In this article we apply a recently established
transference principle in order to obtain the boundedness of
certain functional calculi for semigroup generators. In particular,
it is proved that if $-A$ generates a $C_0$-semigroup on a Hilbert space,
then for each $\tau>0$ the operator $A$ has a bounded calculus 
for the closed ideal of bounded holomorphic functions on a 
(sufficiently large) right half-plane that satisfy 
$f(z)=O(e^{-\tau\Real(z)})$ as $\abs{z}\rightarrow \infty$. 
The bound of this calculus grows at most logarithmically as 
$\tau\searrow 0$. As a consequence, $f(A)$ is a
bounded operator for each holomorphic function $f$ (on 
a right half-plane) with polynomial decay at $\infty$.
Then we show that each semigroup generator has a so-called (strong) $m$-bounded
calculus for all $m\in\N$, and that this property characterizes
semigroup generators. Similar results are obtained if the underlying 
Banach space is a UMD space. Upon restriction to so-called $\gamma$-bounded
semigroups, the Hilbert space results actually hold in general Banach spaces.
\end{abstract}


\end{frontmatter}

\section{Introduction}
\label{Introduction}

Roughly speaking, a functional calculus for a (possibly unbounded) operator
$A$ on a Banach space $X$ is a ``method''  of associating
a closed operator $f(A)$  to each $f= f(z)$ taken from a set of functions
(defined on some subset of the complex plane) in such a way that
formulae valid for the functions turn into valid 
formulae for the operators upon replacing the independent variable $z$ by
$A$. A common way to establish such a calculus is to start
with an algebra  of ``good'' functions $f$ where a definition of $f(A)$
as a bounded operator is more or less straightforward, and then extend
this ``primary'' or ``elementary calculus'' by means of multiplicative
``regularization'' (see \cite[Chapter 1]{Haase1} and \cite{Batty-Haase-Mubeen}). 
It is then  natural
to ask which of the so constructed closed operators $f(A)$ are actually
{\em bounded}, a question particularly relevant in applications, e.g., to 
evolution equations, see for instance \cite{Arendt04, Kalton-Weis1}. 

The latter question links functional calculus theory
to the theory of vector-valued
singular integrals, best seen in the theory
of sectorial (or strip-type) operators with  a bounded
$\HT^\infty$-calculus, see for instance \cite{Kunstmann-Weis}. 
It appears there that in order to obtain nontrivial results
the underlying Banach space must allow
for singular integrals to converge, i.e., be a UMD space (or better, 
a Hilbert space). Furthermore, even if the Banach space
is a Hilbert space, it turns out that simple resolvent estimates are not enough 
for the boundedness of an $\HT^\infty$-calculus \cite[Section 9.1]{Haase1}. 

However, some of the central positive results in that theory --- McIntosh's
theorem \cite{McIntosh}, the Boyadzhiev--deLaubenfels theorem 
\cite{Boyadzhiev-deLaubenfels}
and the Hieber--Pr\"uss theorem \cite{Hieber-Pruss98} --- show that
the presence of a $C_0$-group 
of operators does warrant the boundedness
of certain $\HT^\infty$-calculi. In \cite{Haase5}  
the underlying structure of these results was brought to light, 
namely a {\em transference principle}, 
a factorization of the operators $f(A)$ in terms
of vector-valued Fourier multiplier operators. Finally, 
in \cite{Haase6} it was shown that $C_0$-{\em semi}groups also
allow for such transference principles.

In the present paper, we develop this approach further. We
apply  the general form of the transference
principle for semigroups given in \cite{Haase6} in order to 
obtain bounded functional calculi for generators of $C_0$-semigroups.
These results, in particular Theorems 
\ref{main theorem on functional calculus with exponential decay},
\ref{composition with fractional powers}, and 
\ref{main result on m-bounded calculus},
are proved for  general Banach spaces. However,
they make use of (subalgebras of) the analytic $\Ell^{p}(\R;X)$-Fourier multiplier
algebra (see \eqref{analytic fourier multipliers} below for a definition), and hence are 
useful only if the underlying Banach space has a geometry that
allows for nontrivial Fourier multiplier operators.   
In case $X= H$ is a Hilbert space one obtains 
particularly nice results, which we want to summarize here.
(See Section \ref{m-bounded functional calculus}
for the definition of a strong $m$-bounded calculus.)

\begin{theorem}\label{Hilbert space theorem}
Let $-A$ be the generator of a bounded $C_0$-semigroup $(T(t))_{t\in\R_{+}}$ 
on a Hilbert space $H$ with $M := \sup_{t\in\R_{+}}\norm{T(t)}$.
Then the following assertions hold.

\begin{enumerate}
\item[\rm a)] For $\omega < 0$ and 
$f \in \HT^\infty\!(\Ri_\w)$ one has $f(A)T(\tau) \in \La(H)$
with
\begin{align}\label{f(A)T(tau)-estimate} 
\norm{f(A)T(\tau)} \le c(\tau)M^2 \norm{f}_{\HT^\infty\!(\Ri_\w)},
\end{align}
where $c(\tau) = O(\abs{\log(\tau)})$ as $\tau \searrow 0$, and 
$c(\tau) = O(1)$ as $\tau \to \infty$. 
\item[\rm b)] 
For $\w<0<\alpha$ and $\lambda\in\C$ with $\Real \lambda<0$ there is
$C\geq 0$ such that 
\begin{align}\label{f(A)Aalpha-estimate} 
\norm{f(A)(A-\lambda)^{-\alpha}} \le C M^2  \norm{f}_{\HT^\infty\!(\Ri_\w)}
\end{align}
for all $f \in \HT^\infty\!(\Ri_\w)$. In particular, 
$\dom(A^\alpha) \subseteq \dom(f(A))$.

\item[\rm c)]  $A$ has a strong $m$-bounded $\HT^\infty$-calculus of type $0$
for each $m \in \N$. 
\end{enumerate}
\end{theorem}

\noindent
(See Corollary \ref{specialize to Hilbert 1} for a) and b) and 
Corollary \ref{specialize to Hilbert 2} for c).)

\smallskip

When $X$ is  a UMD space one can derive similar results, stated 
in Section \ref{semigroups on UMD spaces}. In Section 
\ref{gamma-bounded semigroups} we extend 
the Hilbert space results to general Banach spaces by replacing 
the assumption of boundedness of the semigroup by its 
{\em $\gamma$-boundedness}, a concept strongly put forward by 
Kalton and Weis \cite{Kalton-Weis2}. In particular, 
Theorem \ref{Hilbert space theorem} holds true 
for $\gamma$-bounded semigroups on arbitrary Banach spaces with $M$ 
being the $\gamma$-bound of the semigroup.

\bigskip

We stress the fact that in contrast to \cite{Haase1}, where 
{\em sectorial} operators and, accordingly, 
functional calculi on sectors, were considered, 
the present article deals with general semigroup generators and
with functional calculi on {\em half-planes}. (See Section 
\ref{Functional calculus} below.)
The abstract theory of (holomorphic) functional calculi on half-planes
can be found in \cite{Batty-Haase-Mubeen} 
where the notion of an $m$-bounded calculus
(for operators of half-plane type) has been introduced. 
Our Theorem \ref{Hilbert space theorem}.c) is basically contained in that paper 
(it follows directly from \cite[Cor.~6.5 and (7.1)]{Batty-Haase-Mubeen}). 

\medskip

The starting point of the present work was the article \cite{Zwart} 
by Hans Zwart, in particular \cite[Theorem 2.5, 2.]{Zwart}. 
There it is shown that one has an estimate \eqref{f(A)T(tau)-estimate}
with $c(\tau) = O(\tau^{-1/2})$ as $\tau \searrow 0$.  
(The case $\alpha > 1/2$ in \eqref{f(A)Aalpha-estimate} 
is an immediate consequence; however, that case
is essentially trivial, see Lemma \ref{decay implies laplace transform} 
below.) 

In \cite{Zwart} and its sequel paper \cite{Schwenninger-Zwart}
the 
functional calculus for a semigroup generator  is
constructed in a rather unconventional way using ideas from 
systems theory. However, a closer
inspection reveals that transference
(i.e., the factorization over a 
Fourier multiplier) is present there as well, hidden
in the very construction of the functional calculus.

\subsubsection{Notation and terminology}
\label{Notation and terminology}

We write $\N:=\left\{1,2,\ldots\right\}$ for the natural numbers and
$\R_{+}:=[0,\infty)$ for the nonnegative reals. The letters $X$ and $Y$
are used to denote Banach spaces over the complex number field.  The
space of bounded linear operators on $X$ is denoted by
$\La(X)$.  For a closed operator $A$ on $X$ its \emph{domain} is
denoted by $\dom(A)$ and its \emph{range} by $\ran(A)$.  The
\emph{spectrum} of $A$ is $\sigma(A)$ and the \emph{resolvent set}
$\rho(A):=\C\setminus\sigma(A)$.  For all $z\in\rho(A)$ the operator
$R(z,A):=(z-A)^{-1}\in\La(X)$ is the \emph{resolvent} of $A$ at $z$.

For $p\in[1,\infty]$, $\Ell^{p}(\R;X)$ is the Bochner space of
equivalence classes of $X$-valued $p$-Lebesgue integrable 
functions on $\R$. The H\"{o}lder conjugate of $p$ is $p'$, defined by
$\frac{1}{p}+\frac{1}{p'}=1$.  The norm on $\Ell^{p}(\R;X)$ is
usually denoted by $\norm{\cdot}_{p}$.

For $\w\in\R$ and $z\in \C$ we let $e_\w(z) := e^{\w z}$.  By $\eM(\R)$
(resp.~$\eM(\R_{+})$) we denote the space of complex-valued Borel
measures on $\R$ (resp.~$\R_{+}$) with the total variation norm, and
we write $\eM_{\w}(\R_{+})$ for the distributions $\mu$ on $\R_{+}$ of
the form $\mu(\ud{s})=e^{\w s}\nu(\ud s)$ for some
$\nu\in\eM(\R_{+})$. Then $\eM_{\w}(\R_{+})$ is a Banach algebra under
convolution with the norm
\begin{align*}
\norm{\mu}_{\eM_{\w}(\R_{+})}:=\norm{e_{-\w}\mu}_{\eM(\R_{+})}.
\end{align*} 
For $\mu\in\eM_{\w}(\R_{+})$ we let $\supp(\mu)$ be the
topological support of $e_{-\w}\mu$. A function $g$ such that
$e_{-\w}g\in \Ell^{1}(\R_{+})$ is usually identified with its
associated measure $\mu\in \eM_{\w}(\R_{+})$ given by
$\mu(\ud{s})=g(s)\ud{s}$. 
Functions and measures defined on $\R_{+}$ are identified 
with their extensions to $\R$ by setting them equal to zero outside $\R_{+}$.

For an open subset $\Omega\neq \emptyset$ of $\C$  we let
$\HT^{\infty}\!(\Omega)$ be the space of bounded holomorphic functions
on $\Omega$, a unital Banach algebra with respect to the norm
\begin{align*} 
\norm{f}_\infty := 
\norm{f}_{\HT^{\infty}\!(\Omega)}:=\sup_{z\in\Omega}\,\abs{f(z)}\qquad(f\in\HT^{\infty}\!(\Omega)).
\end{align*} 
We shall mainly consider the case where $\Omega$ is equal
to a right half-plane
\begin{align*} 
\Ri_{\w}:=\left\{z\in\C\mid\Real(z)>\w\right\}
\end{align*} 
for some $\w\in\R$ (we write $\C_{+}$ for $\Ri_{0}$). 

For convenience we abbreviate the coordinate function $z \mapsto z$ 
simply by the letter $z$. Under this convention, $f = f(z)$ for a function
$f$ defined on some domain $\Omega \subseteq \C$.

The \emph{Fourier transform} of an $X$-valued tempered distribution
$\Phi$ on $\R$ is denoted by $\F\Phi$. For instance, if
$\mu\in\eM(\R)$ then $\F\mu\in \Ell^{\infty}\!(\R)$ is given by
\begin{align*} 
\F\mu(\xi):=\int_{\R}e^{-i\xi s}\,\mu(\ud{s})\qquad(\xi\in\R).
\end{align*} For $\w\in\R$ and $\mu\in\eM_{\w}(\R_{+})$ we let
$\widehat{\mu}\in\HT^{\infty}\!(\Ri_{\w})\cap \Ce(\overline{\Ri_{\w}})$,
\begin{align*} 
\widehat{\mu}(z):=\int_{0}^{\infty}e^{-zs}\,\mu(\ud{s})\qquad(z\in \Ri_{\w}),
\end{align*} 
be the \emph{Laplace-Stieltjes transform} of $\mu$.

\section{Fourier multipliers and functional calculus}
\label{Fourier multipliers and functional calculus}

We briefly discuss some of the concepts that will be used in what follows.

\subsection{Fourier multipliers}
\label{Fourier multipliers}

We shall need results from Fourier analysis as collected in
\cite[Appendix E]{Haase1}. Fix a Banach space $X$ and let $m\in
\Ell^{\infty}\!(\R;\La(X))$ and $p\in[1,\infty]$. Then $m$ is a
\emph{bounded $\Ell^{p}(\R;X)$-Fourier multiplier} if there exists
$C\geq 0$ such that
\begin{align*} 
T_{m}(\ph):=\mathcal{F}^{-1}(m \cdot \F \ph)\in
\Ell^{p}(\R;X) \quad \text{and} \quad
\norm{T_{m}(\ph)}_{p}\leq C\norm{\ph}_{p}
\end{align*} 
for each $X$-valued Schwartz function $\ph$.  
In this case the mapping $T_{m}$ extends
uniquely to a bounded operator on $\Ell^{p}(\R;X)$
if $p<\infty$ and on $\Ce_{0}(\R;X)$ if $p=\infty$.  
We let $\norm{m}_{\Ma_{p}(X)}$ be
the norm of the operator $T_{m}$ and let $\Ma_{p}(X)$ be the unital
Banach algebra of all bounded $\Ell^{p}(\R;X)$-Fourier multipliers,
endowed with the norm $\norm{\cdot}_{\Ma_{p}(X)}$.

\medskip

For $\w\in\R$ and $p\in[1,\infty]$ we let
\begin{align}\label{analytic fourier multipliers}
\AM_{p}^{X}\!(\Ri_\w) :=  
\left\{f\in
\HT^{\infty}\!(\Ri_{\w})\mid f(\w+i\cdot)\in\Ma_{p}(X)\right\}
\end{align}
 be the \emph{analytic $\Ell^{p}(\R;X)$-Fourier multiplier
algebra} on $\Ri_{\w}$, endowed the norm
\begin{align*}
\norm{f}_{\AM_{p}^{X}}:=
\norm{f}_{\AM_{p}^{X}\!(\Ri_{\w})}:=\norm{f(\w+i\cdot)}_{\Ma_{p}(X)}.
\end{align*} 
Here $f(\w + i \cdot) \in \Ell^{\infty}\!(\R)$ denotes the
{\em trace} of the holomorphic function $f$ on the boundary
$\partial\Ri_\w = \w + i \R$. By classical Hardy space theory,
\begin{align}\label{fm.e.trace} 
f(\w+is):=\lim_{\w'\searrow
\w}f(\w'+is)
\end{align} 
exists for almost all $s\in \R$, with
$\norm{f(\w+i\cdot)}_{\Ell^{\infty}\!(\R)}=\norm{f}_{\HT^{\infty}\!(\Ri_{\w})}$.

\begin{remark}[Important!]\label{important-remark} 
To simplify notation we sometimes omit the reference to the Banach space $X$
and write $\AM_{p}(\Ri_\w)$ instead of $\AM_{p}^{X}\!(\Ri_\w)$ whenever it is convenient.
\end{remark}

The space $\AM_{p}^X\!(\Ri_{\w})$ is a unital Banach algebra,
contractively embedded in $\HT^{\infty}\!(\Ri_{\w})$, and
$\AM_{1}^X\!(\Ri_{\w})=\AM_{\infty}^X\!(\Ri_{\w})$ is contractively embedded
in $\AM_{p}^X\!(\Ri_{\w})$ for all $p\in(1,\infty)$,
cf.~\cite[p.~347]{Haase1}.

For our main results we need two lemmas about the analytic
multiplier algebra.

\begin{lemma}\label{alternative description} 
For every Banach space $X$, all $\w\in\R$ and
$p\in[1,\infty]$,
\begin{align*}
\AM_{p}^{X}\!(\Ri_{\w})=\left\{f\in\HT^{\infty}\!(\Ri_{\w})\left|\sup_{\w'>\w}\norm{f(\w'+i\cdot)}_{\Ma_{p}(X)}<\infty\right.\right\}
\end{align*}
with 
\quad$\norm{f}_{\AM_{p}^{X}\!(\Ri_{\w})}=\sup_{\w'>\w}\norm{f(\w'+i\cdot)}_{\Ma_{p}(X)}$\quad
for all $f\in\AM_{p}^{X}\!(\Ri_{\w})$.
\end{lemma}

\begin{proof} 
Let $\w\in\R$, $p\in[1,\infty]$ and $f\in\AM_{p}(\Ri_{\w})$. For all $\w'>\w$ and $s\in\R$,
\begin{align*}
f(\w'+is)=\frac{\w'-\w}{\pi}\int_{\R}\frac{f(\w-ir)}{(s-r)^{2}+(\w'-\w)^{2}}\,\ud{r}
\end{align*} 
by \cite[Theorem 5.18]{Rosenblum-Rovnyak}. The right-hand
side is the convolution of $f(\w-i\cdot)$ and the Poisson
kernel $P_{\w'-\w}(r):=\frac{\w'-\w}{\pi(r^{2}+(\w'-\w)^{2})}$. Since
$\norm{P_{\w'-\w}}_{\Ell^{1}(\R)}=1$,
\begin{align*} 
\norm{f(\w'+i\cdot)}_{\Ma_{p}(X)}\leq
\norm{f(\w-i\cdot)}_{\Ma_{p}(X)}=\norm{f}_{\AM_{p}^{X}\!(\Ri_{\w})}.
\end{align*} 
The converse follows from \eqref{fm.e.trace} and \cite[Lemma E.4.1]{Haase1}.
\end{proof}

For $\mu\in\eM(\R)$ and $p\in[1,\infty]$ we let $L_{\mu}\in
\La(\Ell^{p}(\R;X))$,
\begin{align}\label{convolution} 
L_{\mu}(f):=\mu\ast f\qquad(f\in \Ell^{p}(\R;X)),
\end{align}
be the convolution operator associated to $\mu$.

\begin{lemma}\label{laplace transform} 
For each $\w\in\R$ the Laplace
transform induces an isometric algebra isomorphism from
$\eM_{\w}(\R_{+})$ onto $\AM_{1}^{\C}(\Ri_{\w}) = \AM_{1}^{X}\!(\Ri_\w)$. Moreover,
\begin{align*}
\norm{\widehat{\mu}}_{\AM_{p}^{X}\!(\Ri_{\w})}=\norm{L_{e_{-\w}\mu}}_{\La(\Ell^{p}(X))}
\end{align*} 
for all $\mu\in\eM_{\w}(\R_{+})$, $p\in[1,\infty]$.
\end{lemma}

\begin{proof} 
The mappings $\mu\mapsto e_{-\w}\mu$ and $f\mapsto
f(\cdot+\w)$ are isometric algebra isomorphisms $\eM_{\w}(\R_{+}) \to
\eM(\R_{+})$ and $\AM_{p}(\Ri_{\w}) \to \AM_{p}(\C_{+})$,
respectively. Hence it suffices to let $\w=0$. The
Fourier transform induces an isometric isomorphism from $\eM(\R)$ onto
$\Ma_{1}(X)$ \cite[p.347, 8.]{Haase1}. If $\mu\in\eM(\R_{+})$
and $f=\widehat{\mu}\in\HT^{\infty}\!(\C_{+})$ then
$f(i\cdot)=\F\mu\in\Ma_{1}(X)$ with
$\norm{f(i\cdot)}_{\Ma_{1}(X)}=\norm{\mu}_{\eM(\R_{+})}$.
Moreover, for $p\in[1,\infty]$,
\begin{align*} 
\norm{f(i\cdot)}_{\Ma_{p}(X)}=\sup_{\norm{g}_{p}\leq
1}\norm{\mathcal{F}^{-1}(f(i\cdot)\mathcal{F}g)}_{p}=\sup_{\norm{g}_{p}\leq
1}\norm{\mu\ast g}_{p}=\norm{L_{\mu}}_{\La(\Ell^{p}(X))}
\end{align*} 
If $f\in\AM_{1}\!(\C_{+})$ then $f(i\cdot)=\F\mu$ for some
$\mu\in\eM(\R)$. An application of Liouville's theorem shows that
$\supp(\mu)\subseteq \R_{+}$, hence $f=\widehat{\mu}$.
\end{proof}

\subsection{Functional calculus}
\label{Functional calculus} 

We assume that the reader is familiar with the basic notions and results of
the theory of $C_0$-semigroups as developed, e.g., in
\cite{Engel-Nagel}, and just recall some facts which will be needed 
in this article.
 
Each $C_0$-semigroup $T = (T(t))_{t\in \R_+}$ on a Banach space $X$
has \emph{type} $(M,\w)$ for some $M\geq1$ and $\w\in\R$, which means
that $\norm{T(t)}\leq Me^{\w t}$ for all $t\geq 0$. 
The \emph{generator} of $T$ is the
unique closed operator $-A$ such that
\begin{align*} 
(\lambda+A)^{-1}x=\int_{0}^{\infty}e^{-\lambda
t}T(t)x\,\ud{t}\qquad(x\in X)
\end{align*} 
for $\Real(\lambda)$ large.
The \emph{Hille-Phillips (functional) calculus} for $A$ is 
defined as follows. Fix
$M\geq 1$ and $\w_{0}\in\R$ such that $T$ has type $(M,-\w_{0})$. For
$\mu\in\eM_{\w_{0}}(\R_{+})$ define $T_{\mu}\in\La(X)$ by
\begin{align}\label{definition operator}
T_{\mu}x:=\int_{0}^{\infty}T(t)x\,\mu(\ud{t})\qquad(x\in X).
\end{align} 
For $f=\widehat{\mu}\in\AM_{1}\!(\Ri_{\w_{0}})$ set
$f(A):=T_{\mu}$. (This is allowed by the injectivity of the Laplace transform,
see Lemma \ref{laplace transform}.) The mapping $f\mapsto f(A)$ is
an algebra homomorphism. In a second step
the definition of $f(A)$ is extended to a larger class of functions via
\emph{regularization}, i.e., 
\begin{align*} 
f(A):=e(A)^{-1}(ef)(A)
\end{align*} 
if there exists $e\in\AM_{1}\!(\Ri_{\w_{0}})$ such that 
$e(A)$ is injective and $ef\in\AM_{1}\!(\Ri_{\w_{0}})$. Then $f(A)$ is a
closed and (in general) unbounded operator on $X$ and the definition
of $f(A)$ is independent of the choice of regularizer $e$. 
The following lemma shows in particular that for $\w < \w_0$ the operator
$f(A)$ is defined for all $f\in\HT^{\infty}\!(\Ri_{\w})$ by virtue of the regularizer
$e(z) = (z - \lambda)^{-1}$, where  $\Real(\lambda)< \w$.

\begin{lemma}\label{decay implies laplace transform} 
Let $\alpha>\frac{1}{2}$, $\lambda\in \C$ and $\w,\w_{0}\in\R$ with $\Real(\lambda)<\w<\w_{0}$.  Then 
\begin{align*} 
f(z)(z-\lambda)^{-\alpha} \in\AM_{1}\!(\Ri_{\w_{0}}) 
\quad \text{for all $f\in\HT^{\infty}\!(\Ri_{\w})$}.
\end{align*}
\end{lemma}
\begin{proof} 
After shifting we may suppose that $\w=0$. Set
$h(z):=f(z)(z-\lambda)^{-\al}$ for $z\in \C_{+}$. Then $h(i\cdot)\in
\Ell^{2}(\R)$ with
\begin{align*} \norm{h(i\cdot)}_{\Ell^{2}(\R)}^{2}\leq
\int_{\R}\frac{\abs{f(is)}^{2}}{\abs{is-\lambda}^{2\al}}\,\ud{s}\leq
\norm{f}_{\HT^{\infty}\!(\C_{+})}^{2}\int_{\R}\frac{1}{\abs{is-\lambda}^{2\al}}\,\ud{s},
\end{align*} 
hence the Paley-Wiener Theorem \cite[Theorem
5.28]{Rosenblum-Rovnyak} implies that $h=\widehat{g}$ for some $g\in
\Ell^{2}(\R_{+})$. Then $e_{-\w_{0}}g\in \Ell^{1}(\R_{+})$ and
$\widehat{e_{-\w_{0}}g}(z)=h(z+\w_{0})$ for $z\in \C_{+}$. Lemma
\ref{laplace transform} yields $h\in \AM_{1}\!(\Ri_{\w_{0}})$ with
\begin{align*}
\norm{h}_{\AM_{1}\!(\Ri_{\w_{0}})}=\norm{h(\cdot+\w_{0})}_{\AM_{1}\!(\C_{+})}=\norm{e_{-\w_{0}}g}_{\Ell^{1}(\R_{+})}.
\end{align*}
\end{proof}

The Hille--Phillips calculus is an extension of the holomorphic
functional calculus for the operators of half-plane type discussed in  
\cite{Batty-Haase-Mubeen}. An
operator $A$ is \emph{of half-plane type} $\w_{0}\in\R$ if
$\sigma(A)\subseteq\overline{\Ri_{\w_{0}}}$ with
\begin{align*} 
\sup_{\lambda\in \C\setminus
\Ri_{\w}}\norm{R(\lambda,A)}<\infty\qquad \text{for all $\w<\w_{0}$}.
\end{align*} 
One can associate operators
$f(A)\in\La(X)$ to certain elementary functions via Cauchy integrals
and regularize as above to extend the definition to all $f\in
\HT^{\infty}\!(\Ri_{\w})$. If $-A$ generates a $C_{0}$-semigroup of
type $(M,-\w_{0})$ then $A$ is of half-plane type $\w_{0}$, and by
combining \cite[Proposition 2.8]{Batty-Haase-Mubeen} and
\cite[Proposition 3.3.2]{Haase1} one sees that for $\w<\w_{0}$ and
$f\in\HT^{\infty}\!(\Ri_{\w})$ the definitions of $f(A)$ via the
Hille-Phillips calculus and the half-plane calculus coincide.

For a proof of the next, fundamental, lemma
see \cite[Theorem 3.1]{Batty-Haase-Mubeen}.

\begin{lemma}[Convergence Lemma]\label{convergence lemma} 
Let $A$ be a
densely defined operator of half-plane type $\w_{0}\in\R$ on a Banach
space $X$. Let $\w<\w_{0}$ and $(f_{j})_{j\in J}\subseteq
\HT^{\infty}\!(\Ri_{\w})$ be a net satisfying the following conditions:
\begin{enumerate}
\item[\rm 1)] $\sup\left\{\abs{f_{j}(z)}\mid z\in \Ri_{\w},j\in
J\right\}<\infty$;
\item[\rm 2)] $f_{j}(A)\in\La(X)$ for all $j\in J$ and $\sup_{j\in
J}\norm{f_{j}(A)}<\infty$;
\item[\rm 3)] $f(z):=\lim_{j\in J}f_{j}(z)$ exists for all $z\in \Ri_{\w}$.
\end{enumerate} Then $f\in \HT^{\infty}\!(\Ri_{\w})$, $f(A)\in\La(X)$,
$f_{j}(A)\rightarrow f(A)$ strongly and
\begin{align*} \norm{f(A)}\leq \limsup_{j\in J} \norm{f_{j}(A)}.
\end{align*}
\end{lemma}

Let $A$ be an operator of half-plane type $\w_{0}$ and
$\w<\w_{0}$. For a Banach algebra $F$ of functions continuously
embedded in $\HT^{\infty}\!(\Ri_{\w})$, we say that $A$ has a
\emph{bounded $F$-calculus} if there exists a constant $C\geq 0$ such
that $f(A)\in\La(X)$ with
\begin{align}\label{bounded calculus} 
\norm{f(A)}_{\La(X)}\leq C\norm{f}_{F}\qquad \text{for all $f\in F$}.
\end{align} 

The operator $-A$ generates a $C_0$-semigroup
$(T(t))_{t\in\R_{+}}$ of type $(M, \omega)$ if and only if $-(A + \omega)$
generates the semigroup $(e^{-\omega t}T(t))_{t\in\R_{+}}$ of type $(M,
0)$. The functional calculi for $A$ and $A+ \omega$ are linked by the
simple composition rule ``$f(A + \omega) = f(\omega + z)(A)$'' 
\cite[Theorem 2.4.1]{Haase1}. 
Henceforth we shall mainly consider bounded semigroups; all results
carry over to general semigroups by shifting.

\section{Functional calculus for semigroup generators}
\label{functional calculus for semigroup generators}

Define the function $\eta:(0,\infty)\times(0,\infty)\times
[1,\infty]\rightarrow \R_{+}$ by
\begin{align}\label{definition log-quantity}
\eta(\alpha,t,q):=\inf\left\{\norm{\psi}_{q}\norm{\ph}_{q'}\mid
\psi\ast\ph\equiv e_{-\alpha}\textrm{ on }[t,\infty)\right\}.
\end{align} 
 
The set on the right-hand side is not empty: choose for instance
$\psi:=\mathbf{1}_{[0,t]}e_{-\alpha}$ and
$\ph:=\frac{1}{t}e_{-\alpha}$. By \ref{growth estimate lemma},
\begin{align*}
\eta(\alpha,t,q)=O(\abs{\log(\alpha t)})\qquad\text{as $\alpha t\rightarrow 0$},
\end{align*}
for $q\in(1,\infty)$. 

For the following result recall the
definitions of the operators $L_{\mu}$ from \eqref{convolution} and
$T_{\mu}$ from \eqref{definition operator}.

\begin{proposition}\label{transference} 
Let $(T(t))_{t\in\R_{+}}$ be a
$C_{0}$-semigroup of type $(M,0)$ on a Banach space $X$. Let
$p\in[1,\infty]$, $\tau,\w>0$ and $\mu\in\eM_{-\w}(\R_{+})$ with
$\supp(\mu)\subseteq [\tau,\infty)$. Then
\begin{align}\label{transference estimate} 
\norm{T_{\mu}}_{\La(X)}\leq
M^{2}\eta(\w,\tau,p)\norm{L_{e_{\w}\mu}}_{\La(\Ell^{p}(X))}.
\end{align}
\end{proposition}
\begin{proof} 
We can factorize $T_\mu$ as $T_{\mu}=P\circ L_{e_{\w}\mu}\circ
\iota$, where
\begin{itemize}
\item $\iota:X\rightarrow \Ell^{p}(\R;X)$ is given by
\begin{align*} 
\iota(x)(s):= \begin{cases} \psi(-s)T(-s)x &
\text{if}\,\, s \leq 0\\ 0 & \text{if}\,\, s > 0
\end{cases} \qquad (x\in X).
\end{align*} 
\item $P:\Ell^{p}(\R;X)\rightarrow X$ is given by
\begin{align*} 
Pf:=\int_{0}^{\infty}\ph(t)T(t)f(t)\,\ud{t}\qquad
(f\in \Ell^{p}(\R;X)).
\end{align*}
\item $\psi\in \Ell^{p}(\R_{+})$ and $\ph\in \Ell^{p'}(\R_{+})$ are
such that $\psi\ast\ph\equiv e_{-\w}$ on $[\tau,\infty)$.
\end{itemize} This is deduced as in the transference principle from
\cite[Section 2]{Haase6}, using that
$\mu=(\psi\ast\ph)e_{\w}\mu$. H\"{o}lder's inequality then implies
\begin{align*} \norm{T_{\mu}}\leq M^{2}\norm{\psi}_{p}
\norm{L_{e_{\w}\mu}}_{\La(\Ell^{p}(X))}\norm{\ph}_{p'},
\end{align*} and taking the infimum over all such $\psi$ and $\ph$
yields \eqref{transference estimate}.
\end{proof}

Now define, for a Banach space $X$, $\w\in\R$, $p\in[1,\infty]$ and $\tau>0$, the space
\begin{align*}
\AM_{p,\tau}^{X}\!(\Ri_{\w}):=\left\{f\in\AM_{p}^{X}\!(\Ri_{\w})\mid
f(z)=O(e^{-\tau\Real(z)}) \textrm{ as } \abs{z}\rightarrow
\infty\right\},
\end{align*} 
endowed with the norm of $\AM_{p}^{X}\!(\Ri_{\w})$.

\begin{lemma}\label{functions of exponential decay} 
For every Banach space $X$, 
$\w\in\R$, $p\in[1,\infty]$ and $\tau>0$ 
\begin{align}\label{description ideal}
\AM_{p,\tau}^{X}\!(\Ri_{\w})=\AM_{p}^{X}\!(\Ri_{\w})\cap
e_{-\tau}\HT^{\infty}\!(\Ri_{\w})=e_{-\tau}\AM_{p}^{X}\!(\Ri_{\w}).
\end{align} 
In particular, $\AM_{p,\tau}^{X}\!(\Ri_{\w})$ is a
closed ideal in $\AM_{p}^{X}\!(\Ri_{\w})$.
\end{lemma}
\begin{proof} 
The first equality in \eqref{description ideal} is
clear, and so is the inclusion $e_{-\tau}\AM_{p}(\Ri_{\w})\subseteq
\AM_{p,\tau}\!(\Ri_{\w})$. Conversely, if $f\in\AM_{p}(\Ri_{\w})\cap
e_{-\tau}\HT^{\infty}\!(\Ri_{\w})$ then $e_{\tau}f\in\AM_{p}(\Ri_{\w})$
since
\begin{align*}
\norm{e^{\tau(\w+i\cdot)}f(\w+i\cdot)}_{\Ma_{p}(X)}=e^{\tau
\w}\norm{f(\w+i\cdot)}_{\Ma_{p}(X)}.
\end{align*} 
Now suppose that $(f_{n})_{n\in\N}\subseteq
\AM_{p,\tau}\!(\Ri_{\w})$ converges to $f\in\AM_{p}(\Ri_{\w})$. The
Maximum Principle implies
$\norm{e_{\tau}f_{n}}_{\HT^{\infty}\!(\Ri_{\w})}=e^{\tau\w}\norm{f_{n}}_{\HT^{\infty}\!(\Ri_{\w})}$,
hence $(e_{\tau}f_{n})_{n\in\N}$ is Cauchy in
$\HT^{\infty}\!(\Ri_{\w})$. Since it converges pointwise to $e_{\tau}f$,
\eqref{description ideal} implies $f\in\AM_{p,\tau}\!(\Ri_{\w})$.
\end{proof}

We are now ready to prove the main result of this section. Note that
the union of the ideals $\AM_{p,\tau}^{X}\!(\Ri_{\w})$ for $\tau>0$ is
dense in $\AM_{p}^{X}\!(\Ri_{\w})$ with respect to pointwise
and bounded convergence of sequences. If there were a single constant independent
of $\tau$ bounding the $\AM_{p,\tau}^{X}\!(\Ri_{\w})$-calculus for all
$\tau$, the Convergence Lemma would imply that $A$ has a bounded
$\AM_{p}^{X}\!(\Ri_{\w})$-calculus, but this is known to be false in general
\cite[Corollary 9.1.8]{Haase1}.

\begin{theorem}\label{main theorem on functional calculus with
exponential decay} 
For each $p\in(1,\infty)$ there exists a constant
$c_{p}\geq 0$ such that the following holds. Let $-A$ generate a
$C_{0}$-semigroup $(T(t))_{t\in\R_{+}}$ of type $(M,0)$ on a Banach space
$X$ and let $\tau,\w>0$. Then $f(A)\in\La(X)$ and
\begin{align*}
 \norm{f(A)} \leq
\begin{cases} c_{p}\,M^{2}\abs{\log(\w
\tau)}\,\norm{f}_{\AM_{p}^{X}} & \text{if}\,\, \w\tau\leq
\text{$\min(\frac{1}{p},\frac{1}{p'})$},\\ 2M^{2}e^{-\w
\tau}\norm{f}_{\AM_{p}^{X}} & \text{if}\,\, \w\tau>
\text{$\min(\frac{1}{p},\frac{1}{p'})$}
\end{cases}
\end{align*} 
for all $f\in \AM_{p,\tau}^{X}\!(\Ri_{-\w})$. In particular, 
$A$ has a bounded $\AM_{p,\tau}^{X}\!(\Ri_{-\w})$-calculus.
\end{theorem}

\begin{proof} 
First consider $f\in \AM_{1,\tau}\!(\Ri_{-\w})$. Let
$\delta_{\tau}\in\eM_{-\w}(\R_{+})$ be the unit point mass at $\tau$. By
Lemmas \ref{functions of exponential decay} and \ref{laplace
transform} there exists $\mu\in\eM_{-\w}(\R_{+})$ such that
$f=e_{-\tau}\widehat{\mu}=\widehat{\delta_{\tau}\ast\mu}$. Since
$\delta_{\tau}\ast\mu\in\eM_{-\w}(\R_{+})$ with
$\supp(\delta_{\tau}\ast\mu)\subseteq [\tau,\infty)$, Proposition
\ref{transference} and Lemma \ref{laplace transform} yield
\begin{align}\label{main functional calculus estimate} 
\norm{f(A)}\leq
M^{2}\eta(\w,\tau,p)\norm{f}_{\AM_{p}^{X}}.
\end{align} 
Now suppose $f\in \AM_{p,\tau}\!(\Ri_{-\w})$ is
arbitrary. For $\epsilon>0$, $k\in\N$ and $z\in\Ri_{-\w}$ set
$g_{k}(z):=\frac{k}{z-\w+k}$ and
$f_{k,\epsilon}(z):=f(z+\epsilon)g_{k}(z+\epsilon)$. Lemma \ref{decay
implies laplace transform} yields $f_{k,\epsilon}\in
\AM_{1,\tau}\!(\Ri_{-\w})$, hence, by what we have already shown,
\begin{align*} 
\norm{f_{k,\epsilon}(A)}\leq
M^{2}\eta(\w,\tau,p)\norm{f_{k,\epsilon}}_{\AM_{p}^{X}}.
\end{align*} 
The inclusion $\AM_{1}\!(\Ri_{-\w})\subseteq
\AM_{p}(\Ri_{-\w})$ is contractive, so Lemma \ref{laplace transform}
implies that $g_{k}\in \AM_{p}(\Ri_{-\w})$ with
\begin{align*}
\norm{g_{k}}_{\AM_{p}^{X}}\leq\norm{g_{k}}_{\AM_{1}}=k\norm{e_{-k}}_{\Ell^{1}(\R_{+})}=1.
\end{align*} 
Combining this with Lemma \ref{alternative description}
yields
\begin{align*} 
\norm{f_{k,\epsilon}}_{\AM_{p}^{X}}&\leq
\norm{f(\cdot+\epsilon)}_{\AM_{p}^{X}}\norm{g_{k}(\cdot+\epsilon)}_{\AM_{p}^{X}}\\
&\leq \norm{f}_{\AM_{p}^{X}}.
\end{align*} 
In particular,
$\sup_{k,\epsilon}\norm{f_{k,\epsilon}}_{\infty}<\infty$
and $\sup_{k,\epsilon}\norm{f_{k,\epsilon}(A)}<\infty$. The
Convergence Lemma \ref{convergence lemma} implies that $f(A)\in
\La(X)$ satisfies \eqref{main functional calculus estimate}. 
\ref{growth estimate lemma} concludes the proof.
\end{proof}

\begin{remark}\label{exponential decay in every direction} 
Because $\AM_{1}\!(\Ri_{-\w})=\AM_{\infty}\!(\Ri_{-\w})$ is contractively embedded
in $\AM_{p}(\Ri_{-\w})$, Theorem \ref{main theorem on functional
calculus with exponential decay} also holds for $p=1$ and
$p=\infty$. However, $A$ trivially has a bounded 
$\AM_{1}$-calculus by Lemma \ref{laplace transform} and 
the Hille-Phillips calculus.

Note that the exponential decay of $\abs{f(z)}$ is
only required as the real part of $z$ tends to infinity. If
$\abs{f(z)}$ decays exponentially as $\abs{z}\rightarrow \infty$ the
result is not interesting, by Lemma \ref{decay implies laplace
transform}.
\end{remark}

We can equivalently formulate Theorem \ref{main theorem on functional
calculus with exponential decay} as a statement about composition with
semigroup operators.

\begin{corollary}\label{composition with semigroup operators} 
Under the assumptions of Theorem \ref{main theorem on functional calculus
with exponential decay}, $f(A)T(\tau)\in \La(X)$ and
\begin{align*} 
\norm{f(A)T(\tau)}\leq \begin{cases} c_{p}\,M^{2}\abs{\log(\w
\tau)}\,e^{\w\tau}\norm{f}_{\AM_{p}^{X}} & \text{if}\,\,
\w\tau\leq \min(\frac{1}{p},\frac{1}{p'}),\\
2M^{2}\norm{f}_{\AM_{p}^{X}} & \text{if}\,\, \w\tau>
\min(\frac{1}{p},\frac{1}{p'})
\end{cases}
\end{align*} 
for all $f\in \AM_{p}^{X}\!(\Ri_{-\w})$.
\end{corollary}

\begin{proof} 
Note that $f(A)T(\tau)=(e_{-\tau}f)(A)$ and
$\norm{e_{-\tau}f}_{\AM_{p}^{X}}=e^{\w\tau}\norm{f}_{\AM_{p}^{X}}$.
\end{proof}

\subsubsection{\textbf{\emph{Additional results}}}
\label{Additional results}

As a first corollary of Theorem \ref{main theorem on functional
calculus with exponential decay} we obtain a sufficient condition for
a semigroup generator to have a bounded $\AM_{p}$-calculus.

\begin{corollary}\label{sufficient for bounded calculus} 
Let $-A$
generate a bounded $C_{0}$-semigroup $(T(t))_{t\in\R_{+}}\subseteq
\La(X)$ with
\begin{align*} 
\bigcup_{\tau>0}\ran(T(\tau))=X.
\end{align*} 
Then $A$ has a bounded $\AM_{p}^{X}\!(\Ri_{\w})$-calculus for
all $\w<0$, $p\in [1,\infty]$.
\end{corollary}

\begin{proof} 
Using Corollary \ref{composition with semigroup
operators}, note that $f(A)T(\tau)\in \La(X)$ implies
$\ran(T(\tau))\subseteq \dom(f(A))$. An application of the Closed
Graph Theorem (using the Convergence Lemma) yields \eqref{bounded
calculus}.
\end{proof}

\begin{theorem}\label{composition with fractional powers} 
Let $p\in(1,\infty)$, $\w>0$ and $\alpha,\lambda\in\C$ with
$\Real(\lambda)<0<\Real(\alpha)$. There exists a constant 
$C=C(p,\alpha,\lambda,\w)\geq 0$ such
that the following holds. Let $-A$ generate a $C_{0}$-semigroup
$(T(t))_{t\in\R_{+}}$ of type $(M,0)$ on a Banach space $X$. Then
$\dom((A-\lambda)^{\al})\subseteq \dom(f(A))$ and
\begin{align*}
\norm{f(A)(A-\lambda)^{-\al}}\leq CM^{2}\norm{f}_{\AM_{p}^{X}}
\end{align*} 
for all $f\in \AM_{p}^{X}\!(\Ri_{-\w})$.
\end{theorem}

\begin{proof} 
First note that $-(A-\lambda)$ generates the exponentially stable
semigroup $(e^{\lambda t}T(t))_{t\in\R_{+}}$. Hence Corollary 3.3.6 in
\cite{Haase1} allows us to write
\begin{align*}
(A-\lambda)^{-\alpha}x=\frac{1}{\Gamma(\alpha)}\int_{0}^{\infty}t^{\alpha-1}e^{\lambda
t}T(t)x\,\ud{t}\qquad(x\in X).
\end{align*} 
Fix $f\in \AM_{p}(\Ri_{-\w})$ and set 
$a:=\frac{1}{\w}\min\left\{\frac{1}{p},\frac{1}{p'}\right\}$. By
Corollary \ref{composition with semigroup operators},
\begin{align*}
\int_{0}^{\infty}t^{\Real(\alpha)-1}e^{\Real(\lambda)t}
\norm{f(A)T(t)x}\ud{t} 
\le C M^{2}\norm{f}_{\AM_{p}^{X}} \norm{x} < \infty
\end{align*}
for all $x\in X$, where
\[ C = c_{p}\int_{0}^{a}t^{\Real(\alpha)-1}\abs{\log(\w
t)}e^{(\Real(\lambda)+\w) t}\,\ud{t}
+2\int_{a}^{\infty}t^{\Real(\alpha)-1}e^{\Real(\lambda) t}\,\ud{t}
\]
is independent of $f$, $M$,  and $x$.
Since $f(A)$ is a closed operator, this implies that 
$(A-\lambda)^{-\alpha}$ maps into $\dom(f(A))$ with
\begin{align*}
f(A)(A-\lambda)^{-\alpha}=\frac{1}{\Gamma(\alpha)}\int_{0}^{\infty}t^{\alpha-1}e^{\lambda
t}f(A)T(t)\,\ud{t}
\end{align*}
as a strong integral.
\end{proof}

\begin{remark}\label{large domains} 
Theorem \ref{composition with fractional powers} shows that 
for each analytic multiplier function $f$ the domain $\dom(f(A))$ 
is relatively large, it contains the real 
interpolation spaces $(X,\dom(A))_{\theta,q}$ and the complex 
interpolation spaces $[X,\dom(A)]_{\theta}$ for all
$\theta\in(0,1)$ and $q\in[1,\infty]$. This follows from
\cite[Proposition 1.1.4]{Lunardi} and \cite[Corollary 6.6.3]{Haase1}
for real interpolation spaces and then from \cite[Proposition 2.1.10]{Lunardi}
for the complex interpolation spaces.
\end{remark}

\begin{remark}\label{range of operators} 
We can describe the range of $f(A)(A-\lambda)^{-\alpha}$ in Theorem
\ref{composition with fractional powers} more explicitly. In fact,
\begin{align*} 
\ran(f(A)(A-\lambda)^{-\alpha})\subseteq
\dom\left((A-\lambda)^{\beta}\right)
\end{align*} 
for all $\Real(\beta)<\Real(\alpha)$. Indeed, this follows if we 
show that $\ran(A-\lambda)^{-\alpha}\subseteq \dom((A-\lambda)^{\beta}f(A))$,
and \cite[Theorem 1.3.2]{Haase1} implies
\begin{align*} 
\dom((A-\lambda)^{\beta}f(A))=\dom(f(A))\cap
\dom\left([(z-\lambda)^{\beta}f(z)](A)\right).
\end{align*} 
The inclusion $\ran((A-\lambda)^{-\alpha})\subseteq
\dom(f(A))$ follows from Theorem \ref{composition with
fractional powers}. Since
\begin{align*} 
[(z-\lambda)^{\beta}f(z)](A)(A-\lambda)^{-\alpha}=
[(z-\lambda)^{\beta-\alpha}f(z)](A)=f(A)(A-\lambda)^{\beta-\alpha},
\end{align*} 
the same holds for the inclusion
$\ran((A-\lambda)^{-\alpha})\subseteq \dom\left([(z-\lambda)^{\beta}f(z)](A)\right)$.
\end{remark}

\subsubsection{\textbf{\emph{{Semigroups on Hilbert spaces}}}}
\label{semigroups on Hilbert spaces}

If $X=H$ is a Hilbert space, Plancherel's Theorem implies 
$\AM_{2}^{H}=\HT^{\infty}$ with equality of norms. Hence the 
theory above specializes to the following result, implying a) 
and b) of Theorem \ref{Hilbert space theorem}.

\begin{corollary}\label{specialize to Hilbert 1} 
Let $-A$ generate a bounded $C_{0}$-semigroup 
$(T(t))_{t\in\R_{+}}$ of type $(M,0)$ on a Hilbert space $H$. 
Then the following assertions hold.

\begin{enumerate}
\item[\rm a)] There exists a universal constant $c\geq 0$ 
such that the following holds. Let $\tau,\w>0$. Then
$f(A)\in \La(H)$ and
\begin{align*}
 \norm{f(A)} \leq\begin{cases} c\,M^{2}\abs{\log(\w
\tau)}\;\norm{f}_{\infty} & \text{if }\,\, \w\tau\leq
\frac{1}{2},\\ 2M^{2}e^{-\w
\tau}\norm{f}_{\infty} & \text{if }\,\, \w\tau>
\frac{1}{2}
\end{cases}
\end{align*} 
for all $f\in e_{-\tau}\HT^{\infty}\!(\Ri_{-\w})$. Moreover,
$f(A)T(\tau)\in \La(H)$ with
\begin{align*}
 \norm{f(A)T(\tau)} \leq\begin{cases} c\,M^{2}\abs{\log(\w
\tau)}e^{\w\tau}\norm{f}_{\infty} & \text{if }\,\, \w\tau\leq
\frac{1}{2},\\ 2M^{2}\norm{f}_{\infty} & \text{if }\,\, \w\tau>
\frac{1}{2}
\end{cases}
\end{align*} 
for all $f\in \HT^{\infty}\!(\Ri_{-\w})$.
\item[\rm b)] If 
\begin{align*} 
{\bigcup}_{\tau>0}\ran(T(\tau))=H,
\end{align*} 
then $A$ has a bounded $\HT^{\infty}\!(\Ri_{\w})$-calculus for
all $\w<0$.
\item[\rm c)] For $\w<0$ and $\alpha,\lambda\in\C$ with $\Real(\lambda)< 0<\Real(\alpha)$ there is
$C=C(\alpha,\lambda,\w)\geq 0$ such that 
\begin{align*} 
\norm{f(A)(A-\lambda)^{-\alpha}} \le C M^2  \norm{f}_{\infty}
\end{align*}
for all $f \in \HT^\infty\!(\Ri_{\w})$. In particular, 
$\dom(A^\alpha) \subseteq \dom(f(A))$.
\end{enumerate}
\end{corollary}

Part c) shows that, even though semigroup generators on Hilbert spaces 
do not have a bounded $\HT^{\infty}$-calculus in general, each function $f$ 
that decays with polynomial rate $\alpha>0$ at infinity yields a bounded 
operator $f(A)$. For $\al>\frac{1}{2}$ this is already covered by Lemma 
\ref{decay implies laplace transform}, but for $\al\in(0,\frac{1}{2}]$ 
it appears to be new.

\begin{remark}\label{numerical stability}
Part c) of Corollary 3.10 yields a statement about stability of numerical methods. Let $-A$ generate an exponentially stable semigroup $(T(t))_{t\geq 0}$ on a Hilbert space, let $r\in\HT^{\infty}(\C_{+})$ be such that $\norm{r}_{\HT^{\infty}(\C_{+})}\leq 1$, and let $\alpha,h>0$. Then 
\begin{align}\label{numerical stability equation}
\sup\left\{\norm{r(hA)^{n}x}\mid n\in\N,x\in \dom(A^{\alpha})\right\}<\infty
\end{align}
follows from c) in Corollary 3.10 after shifting the generator. Elements of the form $r^{n}(hA)x$ are often used in numerical methods to approximate the solution of the abstract Cauchy problem associated to $-A$ with initial value $x$, and \eqref{numerical stability equation} shows that such approximations are stable whenever $x\in\dom(A^{\alpha})$.
\end{remark}

\section{$m$-Bounded functional calculus}
\label{m-bounded functional calculus}

In this section we describe another transference principle for
semigroups, one that provides estimates for the norms of operators of
the form $f^{(m)}(A)$ for $f$ an analytic multiplier function and
$f^{(m)}$ its $m$-th derivative, $m\in\N$.
We use terminology from Section 5 of \cite{Batty-Haase-Mubeen}. Moreover,
we recall our notational simplification $\AM_{p}(\Ri_\w) := \AM_{p}^{X}\!(\Ri_\omega)$
(Remark \ref{important-remark}).

\medskip

Let $\w < \w_0$ be real numbers. 
An operator $A$ of half-plane type $\w_{0}$ 
on a Banach space $X$
has an \emph{$m$-bounded
$\AM_{p}^{X}\!(\Ri_{\w})$-calculus} if there exists $C\geq 0$ such that
$f^{(m)}(A)\in \La(X)$ with
\begin{align*} 
\norm{f^{(m)}(A)}\leq C\norm{f}_{\AM_{p}^{X}} \qquad \text{for all $f\in\AM_{p}^{X}\!(\Ri_{\w})$}.
\end{align*} 
This is well defined
since the Cauchy integral formula implies that $f^{(m)}$ is bounded on
every half-plane $\Ri_{\w'}$ with $\w'>\w$. 

We say that  $A$ has a \emph{strong $m$-bounded
$\AM_{p}^{X}$-calculus of type $\w_{0}$} if $A$ has an $m$-bounded
$\AM_{p}^{X}\!(\Ri_{\w})$-calculus for {\em every} $\w<\w_{0}$ and
such that for some $C\ge 0$ one has
\begin{align}\label{strong m-bounded} 
\norm{f^{(m)}(A)}\leq
\frac{C}{(\w_{0}-\w)^{m}}\norm{f}_{\AM_{p}^{X}\!(\Ri_\w)}
\end{align} 
for all $f\in \AM_{p}^{X}\!(\Ri_{\w})$ and $\w < \w_0$.

\begin{lemma}\label{1-bounded implies m-bounded}
Let $A$ be an operator of half-plane type $\w_{0}\in\R$ on a 
Banach space $X$, and let $p\in[1,\infty]$ and $m\in\N$. 
If $A$ has a strong $m$-bounded $\AM_{p}^{X}$-calculus of type $\w_{0}$, 
then $A$ has a strong $n$-bounded $\AM_{p}^{X}$-calculus of type 
$\w_{0}$ for all $n>m$.
\end{lemma}

\begin{proof}
Let $\w<\alpha<\beta<\w_{0}$, $f\in\AM_{p}(\Ri_{\w})$ and $n\in\N$. Then
\begin{align*}
f^{(n)}(\beta+is)&=\frac{n!}{2\pi i}\int_{\R}\frac{f(\alpha+ir)}
{(\alpha+ir-(\beta+is))^{n+1}}\,\ud r\\
&=\frac{n!}{2\pi i}\left(f(\alpha+i\cdot)\ast 
(\alpha-\beta-i\cdot)^{-n-1}\right)(s)
\end{align*}
for all $s\in\R$, by the Cauchy integral formula. Hence, using 
Lemma \ref{alternative description},
\begin{align*}
\norm{f^{(n)}(\beta+i\cdot)}_{\Ma_{p}(X)}&\leq \frac{n!}{2\pi}
\norm{(\alpha-\beta-i\cdot)^{-n-1}}_{\Ell^{1}(\R)}
\norm{f(\alpha+i\cdot)}_{\Ma_{p}(X)}\\
&\leq \frac{C}{(\beta-\alpha)^{n}}\norm{f}_{\AM_{p}(\Ri_{\w})}
\end{align*}
for some $C=C(n)\geq 0$ independent of $f$, $\beta$, $\alpha$ and $\w$. Letting $\alpha$ tend to $\w$ yields
\begin{align}\label{cauchy inequality}
\norm{f^{(n)}}_{\AM_{p}(\Ri_{\beta})}=\norm{f^{(n)}(\beta+i\cdot)}_{\Ma_{p}(X)}\leq 
\frac{C}{(\beta-\w)^{n}}\norm{f}_{\AM_{p}(\Ri_{\w})}.
\end{align}
Now let $n>m$. Applying \eqref{cauchy inequality} with $n-m$ 
in place of $n$ shows that $f^{(n-m)}\in\AM_{p}(\Ri_{\beta})$ with
\begin{align*}
\norm{f^{(n)}(A)}&\leq \frac{C'}{(\w_{0}-\beta)^{m}}\norm{f^{(n-m)}}_{\AM_{p}(\Ri_{\beta})}\\
&\leq \frac{CC'}{(\w_{0}-\beta)^{m}(\beta-\w)^{n-m}}\norm{f}_{\AM_{p}(\Ri_{\w})}.
\end{align*}
Finally, letting $\beta=\frac{1}{2}(\w+\w_{0})$,
\begin{align*}
\norm{f^{(n)}(A)}\leq \frac{C''}{(\w_{0}-\w)^{n}}\norm{f}_{\AM_{p}(\Ri_{\w})}
\end{align*}
for some $C''\geq 0$ independent of $f$ and $\w$.
\end{proof} 

For the transference principle in Proposition \ref{transference} it
is essential that the support of $\mu\in\eM_{\w}(\R_{+})$ is
contained in some interval $[\tau,\infty)$ with $\tau>0$. In general
one cannot expect to find such a transference principle for arbitrary
$\mu$, as this would allow one to prove that semigroup generators have
a bounded analytic multiplier calculus. But this is known to be false
in general, cf.~\cite[Corollary 9.1.8]{Haase1}. However, if we let
$t\mu$ be given by $(t\mu)(\ud t):=t\mu(\ud t)$ then we can deduce the
following transference principle. We use the conventions $1/\infty:=0$, 
$\infty^{0}:=1$.

\begin{proposition}\label{transference m-bounded} 
Let $-A$ be the generator of  a
$C_{0}$-semigroup $(T(t))_{t\in\R_{+}}$ of type
$(M,0)$ on a Banach space $X$. Let $p\in[1,\infty]$, $\w<0$ and $\mu\in\eM_{\w}(\R_{+})$. Then
\begin{align*} 
\norm{T_{t\mu}}\leq
\frac{M^{2}}{\abs{\w}}p^{-1/p}(p')^{-1/p'}\norm{L_{e_{-\w}\mu}}_{\mathcal{L}(\Ell^{p}(X))}.
\end{align*}
\end{proposition}

\begin{proof} 
We can factorize $T_{t\mu}$ as $T_{t\mu}=P\circ L_{e_{-\w}\mu}\circ
\iota$, where $\iota$ and $P$ are as in the proof of Proposition
\ref{transference} with $\psi,\ph:=\mathbf{1}_{\R_{+}}e_{\w}$. This
follows from the abstract transference principle in \cite[Section
2]{Haase6}, since $(\psi\ast\ph)e_{-\w}\mu=t\mu$. Then
\begin{align*} 
\norm{T_{t\mu}}&\leq M^{2}
\norm{e_{\w}}_{p'}\norm{L_{e_{-\w}\mu}}_{\mathcal{L}(\Ell^{p}(X))}\norm{e_{\w}}_{p}\\
&=\frac{M^{2}}{\abs{\w}}p^{-1/p}(p')^{-1/p'}\norm{L_{e_{-\w}\mu}}_{\mathcal{L}(\Ell^{p}(X))},
\end{align*}
by H\"{o}lder's inequality.
\end{proof}

We are now ready to prove our main result on $m$-bounded functional
calculus, a generalization of Theorem 7.1 in \cite{Batty-Haase-Mubeen}
to arbitrary Banach spaces.

\begin{theorem}\label{main result on m-bounded calculus} 
Let $A$ be a
densely defined operator of half-plane type $0$ on a Banach space
$X$. Then the following assertions are equivalent:
\begin{enumerate}
\item[\rm (i)] $-A$ is the generator of a
bounded $C_{0}$-semigroup on $X$.
\item[\rm (ii)] $A$ has a strong $m$-bounded
$\AM_{p}^{X}$-calculus of type $0$ for some/all $p\in[1,\infty]$ and
some/all $m\in\N$.
\end{enumerate} 
In particular, if $-A$ generates a bounded
$C_{0}$-semigroup then $A$ has an $m$-bounded
$\AM_{p}^{X}\!(\Ri_{\w})$-calculus for all $\w<0$, $p\in[1,\infty]$ and
$m\in\N$.
\end{theorem}

\begin{proof} 
(i) $\Rightarrow$ (ii)
By Lemma \ref{1-bounded implies m-bounded} it suffices to let $m=1$. 
We proceed along the same
lines as the proof of Theorem \ref{main theorem on functional calculus
with exponential decay}. Let $(T(t))_{t\in\R_{+}}\subseteq \La(X)$ be
the semigroup generated by $-A$ and fix $\w<0$, $p\in[1,\infty]$ and
$f\in \AM_{p}(\Ri_{\w})$. Define the functions
$f_{k,\epsilon}:=f(\cdot+\epsilon)g_{k}(\cdot+\epsilon)$ for $k\in\N$
and $\epsilon>0$, where $g_{k}(z):=\frac{k}{z-\w+k}$ for
$z\in\Ri_{\w}$. Then $f_{k,\epsilon}\in\AM_{1}\!(\Ri_{\w})$ by Lemma
\ref{decay implies laplace transform}, and Lemma \ref{laplace
transform} yields $\mu_{k,\epsilon}\in\eM_{\w}(\R_{+})$ with
$f_{k,\epsilon}=\widehat{\mu_{k,\epsilon}}$. Now
\begin{align*} 
f_{k,\epsilon}'(z)&=\lim_{h\rightarrow
0}\frac{f_{k,\epsilon}(z+h)-f_{k,\epsilon}(z)}{h}=\lim_{h\rightarrow
0}\int_{0}^{\infty}\frac{e^{-(z+h)t}-e^{-zt}}{h}\,\mu_{k,\epsilon}(\ud{t})\\ 
&=-\int_{0}^{\infty}te^{-zt}\,\mu_{k,\epsilon}(\ud{t})
=-\widehat{t\mu_{k,\epsilon}}(z)
\end{align*} 
for $z\in\Ri_{\w}$, by the Dominated Convergence Theorem. Hence
$f_{k,\epsilon}'(A)=-T_{t\mu_{k,\epsilon}}$, and Proposition
\ref{transference m-bounded} and Lemma \ref{laplace transform} imply
\begin{align*} 
\norm{f_{k,\epsilon}'(A)}\leq
\frac{M^{2}}{\abs{\w}}\,p^{-1/p}(p')^{-1/p'}\norm{f_{k,\epsilon}}_{\AM_{p}^{X}}.
\end{align*} 
Furthermore,
$\sup_{k,\epsilon}\norm{f_{k,\epsilon}}_{\AM_{p}^{X}}\leq
\norm{f}_{\AM_{p}^{X}}$. In particular, the $f_{k,\epsilon}$ are
uniformly bounded. By the Cauchy integral formula, so are the
derivatives $f'_{k,\epsilon}$ on every smaller half-plane. Since
$f_{k,\epsilon}'(z)\rightarrow f'(z)$ for all $z\in\Ri_{\w}$ as
$k\rightarrow \infty$, $\epsilon\rightarrow 0$, the Convergence Lemma
yields $f'(A)\in\La(X)$ with
\begin{align*} 
||f'(A)||\leq
\frac{M^{2}}{\abs{\w}}\,p^{-1/p}(p')^{-1/p'}\norm{f}_{\AM_{p}^{X}},
\end{align*} 
which is \eqref{strong m-bounded} for $m=1$.

For (ii) $\Rightarrow$ (i) assume that $A$ has a strong $m$-bounded
$\AM_{p}$-calculus of type $0$ for some $p\in[1,\infty]$ and some
$m\in\N$. Then
\begin{align*} 
e_{-t}\in\AM_{1}\!(\Ri_{\w})\subseteq \AM_{p}(\Ri_{\w})
\end{align*} 
for all $t> 0$ and $\w<0$, with
\begin{align*} 
\norm{e_{-t}}_{\AM_{p}\!(\Ri_{\w})}\leq
\norm{e_{-t}}_{\AM_{1}\!(\Ri_{\w})}=e^{-t\w}.
\end{align*} 
Now $(e_{-t})^{(m)}=(-t)^{m}e_{-t}$ implies
\begin{align*} 
t^{m}\norm{e^{-tA}}\leq \frac{C}{\abs{\w}^{m}}e^{-t\w}.
\end{align*} 
Letting $\w:=-\frac{1}{t}$ and using Lemma 2.5 in
\cite{Batty-Haase-Mubeen} yields the required statement.
\end{proof}

If $X=H$ is a Hilbert space then Plancherel's theorem yields the 
following result, which is a generalization of\
Theorem 7.1 in \cite{Batty-Haase-Mubeen}. It contains
part c) of Theorem \ref{Hilbert space theorem}.

\begin{corollary}\label{specialize to Hilbert 2} 
Let $A$ be a densely defined operator of half-plane type $0$ on a 
Hilbert space $H$. Then the following assertions are equivalent:
\begin{enumerate}
\item[\rm (i)] $-A$ is the generator of a
bounded $C_{0}$-semigroup on $H$.
\item[\rm (ii)] $A$ has a strong $m$-bounded
$\HT^{\infty}$-calculus of type $0$ for 
some/all $m\in\N$.
\end{enumerate} 
In particular, if $-A$ generates a bounded
$C_{0}$-semigroup then $A$ has an $m$-bounded
$\HT^{\infty}\!(\Ri_{\w})$-calculus for all $\w<0$ and
$m\in\N$.
\end{corollary}

\section{Semigroups on UMD spaces}
\label{semigroups on UMD spaces}

A Banach space $X$ is a \emph{UMD space} if the function
$t\mapsto \mathrm{sgn}(t)$ is a bounded $\Ell^{2}(X)$-Fourier
multiplier. For $\w\in\R$ we let 
\begin{align*}
\HT^{\infty}_{1}\!(\Ri_{\w}):=\left\{f\in\HT^{\infty}\!(\Ri_{\w})\mid
(z-\w)f'(z)\in\HT^{\infty}\!(\Ri_{\w})\right\}
\end{align*} 
be the \emph{analytic Mikhlin algebra} on $\Ri_{\w}$, a Banach algebra 
endowed with the norm
\begin{align*}
\norm{f}_{\HT^{\infty}_{1}}=\norm{f}_{\HT^{\infty}_{1}\!(\Ri_{\w})}:=\sup_{z\in
\Ri_{\w}}\abs{f(z)}+\abs{(z-\w)f'(z)}\quad\quad(f\in\HT_{1}^{\infty}\!(\Ri_{\w})).
\end{align*}
The vector-valued Mikhlin multiplier theorem 
\cite[Theorem E.6.2]{Haase1} and Lemma \ref{alternative description} 
yield the continuous inclusion
\begin{align*}
\HT^{\infty}_{1}\!(\Ri_{\w})\hookrightarrow
\AM_{p}^{X}\!(\Ri_{\w})
\end{align*}
for each $p\in(1,\infty)$, if $X$ is a UMD space. Combining this with 
Theorems \ref{main theorem on functional calculus with exponential decay} and
\ref{main result on m-bounded calculus} and Corollaries 
\ref{composition with semigroup operators} and
\ref{sufficient for bounded calculus} proves the following theorem.

\begin{theorem}\label{UMD space theorem}
Let $-A$ generate a $C_{0}$-semigroup $(T(t))_{t\in\R_{+}}$ of type $(M,0)$ 
on a UMD space $X$. Then the following assertions hold. 

\begin{enumerate}
\item[\rm a)] 
For each $p\in(1,\infty)$ there exists a constant $c_{p}=c(p,X)\geq 0$ 
such that the following holds. Let $\tau,\w>0$. 
Then $f(A)\in\La(X)$ with
\begin{align*}
\norm{f(A)} \leq\begin{cases} c_{p}\,M^{2}\abs{\log(\w
\tau)}\;\norm{f}_{\HT^{\infty}_{1}} & \text{if }\,\, \w\tau\leq
\min\left\{\frac{1}{p},\frac{1}{p'}\right\},\\ 2c_{p}M^{2}e^{-\w
\tau}\norm{f}_{\HT^{\infty}_{1}} & \text{if }\,\, \w\tau>
\min\left\{\frac{1}{p},\frac{1}{p'}\right\}
\end{cases}
\end{align*}
for all $f\in\HT^{\infty}_{1}\!(\Ri_{-\w})\cap e_{-\tau}\HT^{\infty}\!(\Ri_{-\w})$,
and $f(A)T(\tau)\in \La(X)$ with 
\begin{align*}
 \norm{f(A)T(\tau)} \leq\begin{cases} c_{p}\,M^{2}\abs{\log(\w
\tau)}e^{\w\tau}\norm{f}_{\HT^{\infty}_{1}} & \text{if }\,\, \w\tau\leq
\min\left\{\frac{1}{p},\frac{1}{p'}\right\},\\
 2c_{p}M^{2}\norm{f}_{\HT^{\infty}_{1}} & \text{if }\,\, \w\tau>
\min\left\{\frac{1}{p},\frac{1}{p'}\right\}
\end{cases}
\end{align*} 
for all $f\in \HT^{\infty}_{1}\!(\Ri_{-\w})$.
\item[\rm b)] If 
\begin{align*} 
{\bigcup}_{\tau>0}\ran(T(\tau))=X,
\end{align*} 
then $A$ has a bounded $\HT^{\infty}_{1}\!(\Ri_{\w})$-calculus for
all $\w<0$.
\item[\rm c)] $A$ has a strong $m$-bounded $\HT^{\infty}_{1}$-calculus 
of type $0$ for all $m\in\N$.
\end{enumerate}
\end{theorem}

\begin{remark}\label{domain inclusion for UMD spaces} 
Theorem \ref{composition with fractional powers} yields the domain 
inclusion $\dom(A^\alpha) \subseteq \dom(f(A))$ for all $\alpha\in\C_{+}$,
$\w<0$ and $f\in\HT^{\infty}_{1}(\Ri_{\w})$, on a UMD space $X$. 
However, this inclusion in fact holds true on a general Banach space $X$. 
Indeed, for $\lambda\in\C$ with $\Real(\lambda)<0$, Bernstein's Lemma 
\cite[Proposition 8.2.3]{Arendt-Batty-Hieber-Neubrander} implies 
$\frac{f(z)}{(\lambda-z)^{\alpha}}\in\AM_{1}\!(\C_{+})$, hence
$f(A)(\lambda-A)^{-\alpha}\in\La(X)$ and $\dom(A^\alpha) \subseteq \dom(f(A))$. 
An estimate
\begin{align*} 
\norm{f(A)(\lambda -A)^{-\alpha}} \le C\norm{f}_{\HT_{1}^\infty(\Ri_{\w})}
\end{align*}
then follows from an application of the Closed Graph Theorem and the 
Convergence Lemma.
\end{remark}

\begin{remark}\label{inclusion in the mikhlin algebra remark}
To apply Theorem \ref{UMD space theorem} one can use the continuous inclusion 
\begin{align}\label{inclusion in mikhlin algebra}
\HT^{\infty}\!(\Ri_{\w}\cup(S_{\ph}+a))\subseteq \HT^{\infty}_{1}\!(\Ri_{\w'})
\end{align}
for $\w'>\w$, $a\in\R$ and $\ph\in(\pi/2,\pi]$. Here $\Ri_{\w}\cup(S_{\ph}+a)$ is the union of $\Ri_{\w}$ and 
the translated sector $S_{\ph}+a$, where
\begin{align*}
S_{\ph}:=\left\{z\in\C\mid \abs{\mathrm{arg}(z)}<\ph\right\}.
\end{align*}
Indeed, to derive \eqref{inclusion in mikhlin algebra} it suffices to let $a=0$,
and using Cauchy's integral formula as in \cite[Lemma 8.2.6]{Haase1} yields 
the desired result.
\end{remark}

\section{$\gamma$-Bounded semigroups}
\label{gamma-bounded semigroups}

The geometry of the underlying Banach space $X$ played an essential
role in the results of Sections \ref{functional calculus for semigroup
generators} and \ref{m-bounded functional calculus} in the form of properties
of the analytic multiplier algebra $\AM_{p}^{X}$. To wit, 
in order to identify non-trivial
functions in $\AM_{p}^{X}$ one needs a geometric assumption on  $X$, 
for instance that it is a Hilbert or a UMD space. In this section we take a
different approach and make additional assumptions on the semigroup
instead of the underlying space. We show that if the semigroup in
question is {\em $\gamma$-bounded} then one can recover the Hilbert space
results on an arbitrary Banach space $X$.

For this section we assume the reader to be familiar with the basics
of the theory of $\gamma$-radonifying operators and
$\gamma$-boundedness as collected in the survey article by van Neerven
\cite{vanNeerven}. We use terminology and results from
\cite{Haase6}.

Let $H$ be a Hilbert space and $X$ a Banach space. A linear operator
$T:H\rightarrow X$ is \emph{$\gamma$-summing} if
\begin{align*}
\norm{T}_{\gamma}:=\sup_{F}\left(\mathbb{E}\norm{\sum_{h\in
F}\gamma_{h} Th }_{X}^{2}\right)^{1/2}<\infty,
\end{align*} 
where the supremum is taken over all finite orthonormal
systems $F\subseteq H$ and $(\gamma_{h})_{h\in F}$ is an independent
collection of complex-valued standard Gaussian random variables on
some probability space. Endow
\begin{align*} 
\gamma_{\infty}(H;X):=\left\{T:H\rightarrow X\mid T
\textrm{ is }\gamma\textrm{-summing}\right\}
\end{align*}
 with the norm $\norm{\cdot}_{\gamma}$ and let the space
$\gamma(H;X)$ of all \emph{$\gamma$-radonifying operators} be the
closure in $\gamma_{\infty}(H;X)$ of the finite-rank operators
$H\otimes X$.

For a measure space $(\Omega,\mu)$  let $\gamma(\Omega;X)$
(resp.~$\gamma_{\infty}(\Omega;X)$) be the space of all weakly 
$\Ell^{2}$-functions $f:\Omega\rightarrow X$ for which the integration operator
$J_{f}:\Ell^{2}(\Omega)\rightarrow X$,
\begin{align*} 
J_{f}(g):=\int_{\Omega}g\cdot f\,
\ud\mu\quad\quad\quad(g\in \Ell^{2}(\Omega)),
\end{align*} 
is $\gamma$-radonifying ($\gamma$-summing), and endow it
with the norm $\norm{f}_{\gamma}:=\norm{J_{f}}_{\gamma}$.

A collection $\mathcal{T}\subseteq \La(X)$ is \emph{$\gamma$-bounded}
if there exists a constant $C\geq 0$ such that
\begin{align*}
\left(\mathbb{E}\norm{\sum_{T\in\mathcal{T}'}\gamma_{T}Tx_{T}}^{2}\right)^{1/2}\leq
C\left(\mathbb{E}\norm{\sum_{T\in\mathcal{T}'}\gamma_{T}x_{T}}^{2}\right)^{1/2}
\end{align*} 
for all finite subsets $\mathcal{T}'\subseteq
\mathcal{T}$, sequences $(x_{T})_{T\in\mathcal{T}'}\subseteq X$ and
independent complex-valued standard Gaussian random variables
$(\gamma_{T})_{T\in\mathcal{T}'}$. The smallest such $C$ is the
\emph{$\gamma$-bound} of $\mathcal{T}$ and is denoted by $\llbracket
T\rrbracket^{\gamma}$. Every $\gamma$-bounded collection is uniformly
bounded with supremum bound less than or equal to the $\gamma$-bound,
and the converse holds if $X$ is a Hilbert space.

An important result involving $\gamma$-boundedness is the
\emph{multiplier theorem}. We state a version that is tailored to our purposes.
Given a Banach space $Y$, a function $g:\R\rightarrow Y$ is \emph{piecewise
$W^{1,\infty}$} if $g\in W^{1,\infty}(\R\setminus\left\{a_{1},\ldots,
a_{n}\right\};Y)$ for some finite set $\left\{a_{1},\ldots,
a_{n}\right\}\subseteq \R$.

\begin{theorem}[Multiplier Theorem]\label{multiplier theorem} 
Let $X$
and $Y$ be Banach spaces and $T:\R\rightarrow \La(X,Y)$ a strongly
measurable mapping such that
\begin{align*} 
\mathcal{T}:=\left\{T(s)\mid s\in\R\right\}
\end{align*} 
is $\gamma$-bounded. Suppose furthermore that there
exists a dense subset $D\subseteq X$ such that $s\mapsto T(s)x$ is
piecewise $W^{1,\infty}$ for all $x\in D$. Then the multiplication
operator
\begin{align*} 
\mathcal{M}_{T}:\Ell^{2}(\R)\otimes X\rightarrow
\Ell^{2}(\R;Y)\quad\quad\quad \mathcal{M}_{T}(f\otimes
x)=f(\cdot)T(\cdot)x
\end{align*} 
extends uniquely to a bounded operator
\begin{align*} 
\mathcal{M}_{T}:\gamma(\Ell^{2}(\R);X)\rightarrow
\gamma(\Ell^{2}(\R);Y)
\end{align*} 
with $\norm{\mathcal{M}_{T}}\leq \llbracket
\mathcal{T}\rrbracket^{\gamma}$.
\end{theorem}

\begin{proof} 
That $\mathcal{M}_{T}$ extends uniquely to a bounded
operator into $\gamma_{\infty}(\Ell^{2}(\R);Y)$ with
$\norm{\mathcal{M}_{T}}\leq \llbracket \mathcal{T}\rrbracket^{\gamma}$ is
the content of Theorem 5.2 in \cite{vanNeerven}. To see that in fact
$\textrm{ran}(\mathcal{M}_{T})\subseteq \gamma(\R;Y)$ we employ a
density argument. For $x\in D$ let $a_{1},\ldots, a_{n}\in \R$ be such
that $s\mapsto T(s)x\in W^{1,\infty}(\R\setminus\left\{a_{1},\ldots,
a_{n}\right\};Y)$, and set $a_{0}:=-\infty$, $a_{n+1}:=\infty$. Let
$f\in \Ce_{c}(\R)$ be given and note that
\begin{align*}
\int_{a_{j}}^{a_{j+1}}\norm{f}_{\Ell^{2}(s,a_{j+1})}\norm{T(s)'x}\,\ud{s}<\infty
\end{align*} 
for all $j\in \left\{1,\ldots, n\right\}$. Furthermore,
\begin{align*}
\int_{-\infty}^{a_{1}}\norm{f}_{\Ell^{2}(-\infty,s)}\norm{T(s)'x}\,\ud{s}<\infty.
\end{align*} 
Corollary 6.3 in \cite{Haase6} yields
$(\mathbf{1}_{(a_{j},a_{j+1})}f)(\cdot)T(\cdot)x\in \gamma(\R;Y)$
for all $0\leq j\leq n$,
hence $f(\cdot)T(\cdot)x\in \gamma(\R;Y)$. Since $\Ce_{c}(\R)\otimes
D$ is dense in $\Ell^{2}(\R)\otimes X$, which in turn is dense in
$\gamma(\Ell^{2}(\R);X)$, the result follows.
\end{proof}

We are now ready to prove a generalization of part a) of
Corollary \ref{specialize to Hilbert 1}. Recall that
\begin{align*}
e_{-\tau}\HT^{\infty}\!(\Ri_{\w})=\left\{f\in\HT^{\infty}\!(\Ri_{\w})\mid
f(z)=O(e^{-\tau \Real(z)})\textrm{ as }\abs{z}\rightarrow
\infty\right\}
\end{align*} 
for $\tau>0$, $\w\in\R$.

\begin{theorem}\label{main gamma-result on bounded calculus} 
There exists a universal constant $c\geq 0$ such that the following
holds. Let $-A$ generate a $\gamma$-bounded $C_{0}$-semigroup
$(T(t))_{t\in\R_{+}}$ with $M:=\llbracket
T\rrbracket^{\gamma}$ on a Banach space $X$, and let $\tau,\w>0$. 
Then $f(A)\in\La(X)$ with
\begin{align*}
 \norm{f(A)} \leq
\begin{cases} c\,M^{2}\abs{\log(\w
\tau)}\;\norm{f}_{\infty} & \text{if}\,\, \w\tau\leq
\text{$\frac{1}{2}$} \\ 2M^{2}e^{-\w
\tau}\norm{f}_{\infty} & \text{if}\,\, \w\tau>
\text{$\frac{1}{2}$}
\end{cases}
\end{align*} 
for all $f\in e_{-\tau}\HT^{\infty}\!(\Ri_{-\w})$.

In particular, $A$ has a bounded $e_{-\tau}\HT^{\infty}\!(\Ri_{-\w})$-calculus.
\end{theorem}

\begin{proof} 
We only need to show that the estimate
\eqref{transference estimate} in Proposition \ref{transference} can be
refined to
\begin{align}\label{gamma transference} 
\norm{T_{\mu}}\leq
M^{2}\eta(\w,\tau,2)\norm{L_{e_{\w}\mu}}_{\La(\gamma(\R;X))}
\end{align} 
for $\mu\in\eM_{-\w}(\R_{+})$ with
supp$(\mu)\subseteq [\tau,\infty)$. Then one uses that
\begin{align*}
\norm{L_{e_{\w}\mu}}_{\La(\gamma(\R;X))}\leq 
\norm{\widehat{e_{\w}\mu}}_{\HT^{\infty}\!(\C_{+})}
=\norm{\widehat{\mu}}_{\HT^{\infty}\!(\Ri_{-\w})},
\end{align*}
 by the ideal property of $\gamma(\Ell^{2}(\R);X)$
\cite[Theorem 6.2]{vanNeerven}, and proceeds as in the proof of Theorem
\ref{main theorem on functional calculus with exponential decay} to
deduce the desired result.

To obtain \eqref{gamma transference} we factorize $T_\mu$ as $T_{\mu}=P\circ
L_{e_{\w}\mu}\circ \iota$, where $\iota:X\rightarrow \gamma(\R;X)$ 
and $P:\gamma(\R;X)\rightarrow X$ are given by
\begin{align*} 
\iota x(s):=\psi(-s)T(-s)x\quad\quad&(x\in X, s\in\R),\\
Pg:=\int_{0}^{\infty}\ph(t)T(t)g(t)\,\ud{t}\quad\quad&(g\in \gamma(\R;X)),
\end{align*} 
for $\psi,\ph\in \Ell^{2}(\R_{+})$ such that
$\psi\ast\ph\equiv e_{-\w}$ on $[\tau,\infty)$. This factorization
follows as in Section 2 of \cite{Haase6} once we show that the maps
$\iota$ and $P$ are well-defined and bounded. To this end, first note
that $s\mapsto T(-s)x$ is piecewise $W^{1,\infty}$ for all $x$ in the
dense subset $\dom(A)\subseteq X$ and that
\begin{align*} 
\psi(-\cdot)\otimes x\in \Ell^{2}(-\infty,0)\otimes
X\subseteq \gamma(\Ell^{2}(\R);X).
\end{align*} 
Therefore Theorem \ref{multiplier theorem} yields $\iota
x\in \gamma(\R,X)$ with
\begin{align*} 
\norm{\iota x}_{\gamma}=\norm{J_{\iota x}}_{\gamma}\leq
M\norm{\psi(-\cdot)\otimes
x}_{\gamma}=M\norm{\psi}_{\Ell^{2}(\R_{+})}\norm{x}_{X}.
\end{align*} 
As for $P$, write
\begin{align*} 
Pg=\int_{0}^{\infty}\ph(t)T(t)g(t)\,\ud t=J_{Tg}(\ph)
\end{align*} 
and use Theorem \ref{multiplier theorem} once again to
see that $Tg\in \gamma(\R;X)$. Hence
\begin{align*} 
\norm{Pg}_{X}\leq
\norm{J_{Tg}}_{\gamma}\norm{\ph}_{\Ell^{2}(\R_{+})}\leq
M\norm{\ph}_{\Ell^{2}(\R_{+})}\norm{g}_{\gamma}.
\end{align*} 
Finally, estimating the norm of $T_{\mu}$ through
this factorization and taking the infimum over all $\psi$ and $\ph$
yields \eqref{gamma transference}.
\end{proof}

\begin{corollary}\label{gamma-version corollaries} 
Corollary \ref{specialize to Hilbert 1} generalizes to 
$\gamma$-bounded semigroups on arbitrary Banach spaces upon replacing 
the uniform bound $M$ of $T$ by $\llbracket T\rrbracket^{\gamma}$.
\end{corollary}

Theorem \ref{main result on m-bounded calculus} can be extended in an
almost identical manner to a $\gamma$-version.

\begin{theorem}\label{gamma-version m-bounded}
Let $-A$ generate a $\gamma$-bounded $C_{0}$-semigroup on 
a Banach space $X$. Then $A$ has a strong $m$-bounded 
$\HT^{\infty}$-calculus of type $0$ for all $m\in\N$.
\end{theorem}

\appendix

\section{Growth estimates}
\label{Growth estimates}

In this appendix we examine the function
$\eta:(0,\infty)\times(0,\infty)\times [1,\infty]\rightarrow \R_{+}$
from \eqref{definition log-quantity}:
\begin{align*}
\eta(\alpha,t,q):=\inf\left\{\norm{\psi}_{q}\norm{\ph}_{q'}\mid
\psi\ast\ph\equiv e_{-\alpha}\textrm{ on }[t,\infty)\right\}.
\end{align*} 
We will use the notation $f\lesssim g$ for real-valued
functions $f,g:Z\rightarrow \R$ on some set $Z$ to indicate that there
exists a constant $c\geq 0$ such that $f(z)\leq cg(z)$ for all $z\in
Z$.

\begin{lemma}\label{growth estimate lemma} 
For each $q\in(1,\infty)$
there exist constants $c_{q},d_{q}\geq 0$ such that
\begin{align}\label{growth estimate 1}
 d_{q}\abs{\log(\al
t)}\leq\eta(\al,t,q)\leq c_{q}\abs{\log(\al t)}
\end{align} 
if $\al t\leq
\min\left\{\frac{1}{q},\frac{1}{q'}\right\}$. If $\al
t>\min\left\{\frac{1}{q},\frac{1}{q'}\right\}$ then
\begin{align}\label{growth estimate 2} 
e^{-\al t}\leq\eta(\al,t,q)\leq
2e^{-\al t}.
\end{align}
\end{lemma}

\begin{proof} First note that $\eta(\al,t,q)=\eta(\al
t,1,q)=\eta(1,\al t,q)$ for all $\alpha$, $t$ and $q$. 
Indeed, for $\psi\in \Ell^{q}(\R_{+})$, 
$\ph\in \Ell^{q'}(\R_{+})$ with $\psi\ast\ph\equiv e_{-\alpha}$ on
$[1,\infty)$ define $\psi_{t}(s):=\frac{1}{t^{1/q}}\psi(s/t)$ and
$\ph_{t}(s):=\frac{1}{t^{1/q'}}\ph(s/t)$ for $s\geq 0$. Then
\begin{align*}
\psi_{t}\ast\ph_{t}(r)=\int_{0}^{\infty}\psi\left(\tfrac{r-s}{t}\right)
\ph\left(\tfrac{s}{t}\right)\tfrac{\ud s}{t}
=\psi\ast\ph\left(\tfrac{r}{t}\right)
\end{align*} 
for all $r\geq 0$, so $\psi_{t}\ast\ph_{t}\equiv
e_{-\alpha}$ on $[t,\infty)$. Moreover,
\begin{align*}
\norm{\psi_{t}}_{q}^{q}=\int_{0}^{\infty}\abs{\psi(\tfrac{s}{t})}^{q}
\,\tfrac{\ud s}{t}=\int_{0}^{\infty}\abs{\psi(s)}^{q}\,\ud s=\norm{\psi }_{q}^{q},
\end{align*} 
and similarly $\norm{\ph_{t}}_{q'}=\norm{\ph}_{q'}$. Hence
$\eta(\al,t,q)\leq \eta(\al t,1,q)$. Considering $\psi_{1/t}$ and
$\ph_{1/t}$ yields $\eta(\al,t,q)=\eta(\al t,1,q)$. The other equality
follows immediately. Hence,  to prove any of the inequalities in 
\eqref{growth estimate 1} or \eqref{growth estimate 2}, we can assume 
either that $\alpha=1$ or that $t=1$ (but not both).

For the left-hand inequalities, we assume that $\al=1$ and we first 
consider the left-hand inequality of \eqref{growth estimate 1}. Let $t<1$ and
$\psi\in \Ell^{q}(\R_{+})$, $\ph\in \Ell^{q'}(\R_{+})$ such that
$\psi\ast \ph\equiv e_{-1}$ on $[t,\infty)$. Then
\begin{align*} 
\abs{\log(t)}&=-\log(t)=\int_{t}^{1}\frac{\ud s}{s}\leq
e\int_{t}^{1}e^{-s}\frac{\ud
s}{s}=e\int_{t}^{1}\abs{\psi\ast\ph(s)}\frac{\ud s}{s}\\
&\leq e\int_{t}^{1}\int_{0}^{s}\abs{\psi(s-r) }\cdot\abs{\ph(r)}\,\ud
r\,\frac{d s}{s}\\ &\leq
e\int_{0}^{\infty}\int_{r}^{\infty}\frac{\abs{\ph(s-r)}}{s}\,\ud s
\,\abs{\psi(r)}\,\ud r\\
&=e\int_{0}^{\infty}\int_{0}^{\infty}\frac{\abs{\psi(r)}\abs{\ph(s)}}{s+r}\,\ud
s\,d r\leq \frac{e\pi}{\sin(\pi/q)}\norm{\psi}_{q}\norm{\ph}_{q'},
\end{align*} 
where we used Hilbert's absolute inequality \cite[Theorem
5.10.1]{Garling}. It follows that
\begin{align*} 
\eta(1,t,q)\geq \frac{\sin(\pi/q)}{e\pi}\abs{\log(t)}.
\end{align*} 
For the left-hand inequality of \eqref{growth estimate 2}, we assume 
that $\alpha=1$ and let $t>0$ be arbitrary. Then
\begin{align*}
e^{-t}=(\psi\ast\ph)(t)\leq\int_{0}^{t}\abs{\psi(t-s)}\abs{\ph(s)}\,\ud
s\leq \norm{\psi}_{q}\norm{\ph}_{q'}
\end{align*} 
by H\"{o}lder's inequality, hence $e^{-t}\leq\eta(1,t,q)$.

For the right-hand inequalities in \eqref{growth estimate 1} and 
\eqref{growth estimate 2}, we assume that $t=1$ and first consider 
the right-hand inequality in \eqref{growth estimate 1} for 
$\alpha\leq \min\left\{\frac{1}{q},\frac{1}{q'}\right\}$. 
In the proof of Lemma A.1 in \cite{Haase6} it is shown that
\begin{align*} 
(\psi_{0}\ast
\ph_{0})(s)=\left\{\begin{array}{ll}s, & s\in[0,1)\\ 1,  &s\geq
1\end{array}\right.
\end{align*} 
for
\begin{align*}
\psi_{0}:=\sum_{j=0}^{\infty}\beta_{j}\mathbf{1}_{(j,j+1)}\quad\quad\mathrm{and}\quad\quad
\ph_{0}:=\sum_{j=0}^{\infty}\beta_{j}'\mathbf{1}_{(j,j+1)},
\end{align*} 
where $(\beta_{j})_{j}$ and $(\beta_{j}')_{j}$
are sequences of positive scalars such that
$\beta_{j}=O((1+j)^{-1/q})$ and $\beta_{j}'=O((1+j)^{-1/q'})$
as $j\rightarrow \infty$. Let $\psi:=e_{-\alpha}\psi_{0}$ and
$\ph:=e_{-\alpha}\ph_{0}$. Then $\psi\ast\ph\equiv e_{-\alpha}$ on
$[1,\infty)$ and
\begin{align*} 
\norm{\psi}_{q}^{q}=\norm{e_{-\alpha}\psi_{0}}_{q}^{q}&=
\sum_{j=0}^{\infty}\beta_{j}^{q}\int_{j}^{j+1}e^{-\al qs}\,\ud s
\lesssim \sum_{j=0}^{\infty}\frac{e^{-\al q j}}{1+j}\\ &\leq
1+\int_{0}^{\infty}\frac{e^{-\al qs}}{1+s}\,\ud s=1+e^{\al q}\int_{\al
q}^{\infty}\frac{e^{-s}}{s}\,\ud s.
\end{align*} 
The constant in the first inequality depends only on
$q$. Since $\al q\leq 1$,
\begin{align*} 
\norm{\psi}_{q}^{q}&\lesssim 1+e^{\al q}\left(\int_{\al
q}^{1}\frac{e^{-s}}{s}\,\ud s+\int_{1}^{\infty}\frac{e^{-s}}{s}\,\ud
s\right)\leq 1+\int_{\al q}^{1}\frac{1}{s}\,\ud s+e^{\al
q}\int_{1}^{\infty}e^{-s}\,\ud s\\ &=1-\log(\al q)+e^{\al q-1}\leq
\log\left(\tfrac{1}{\al}\right)+2.
\end{align*} 
Moreover,$\tfrac{1}{\al}\geq q>1$ hence
$\log\left(\tfrac{1}{\al}\right)\geq \log(q)>0$ and
\begin{align*} 
\log\left(\tfrac{1}{\al}\right)+2\leq
\left(1+\frac{2}{\log(q)}\right)\log\left(\tfrac{1}{\al}\right).
\end{align*} 
Therefore
\begin{align*} 
\norm{\psi}_{q}\lesssim
\log\left(\tfrac{1}{\al}\right)^{1/q}=|\log(\al)|^{1/q},
\end{align*} 
for a constant depending only on $q$. In a similar manner
we deduce
\begin{align*} 
\norm{\ph}_{q'}\lesssim |\log(\al)|^{1/q'}
\end{align*} 
for a constant depending only on $q'$ (and thus on
$q$). This yields \eqref{growth estimate 1}.

For the right-hand side of \eqref{growth estimate 2} we assume 
that $t=1$ and, without loss of generality (since $\eta(\al,t,q)=\eta(\al,t,q')$), 
that $\alpha>\frac{1}{q}$. 
Let $\ph:=\mathbf{1}_{[0,1]}e_{\alpha(q-1)}$ and $\psi:=\frac{\alpha q}{e^{\alpha q}-1}\mathbf{1}_{\R_{+}}e_{-\alpha}$. Then
\begin{align*}
\psi\ast\ph(r)=\frac{\alpha q}{e^{\alpha q}-1}\int_{0}^{1}e^{\alpha(q-1)s}e^{-\alpha(r-s)}\,\ud s=e^{-\alpha r}
\end{align*}
for $r\geq 1$. Hence
\begin{align*}
\eta(\alpha,1,q)&\leq \norm{\psi}_{q}\norm{\ph}_{q'}=\frac{\alpha q}{e^{\alpha q}-1}\left(\int_{0}^{\infty}e^{-\alpha q s}\,\ud s\right)^{1/q}\left(\int_{0}^{1}e^{\alpha(q-1)q's}\,\ud s\right)^{1/q'}\\
&=\frac{(\alpha q)^{(q-1)/q}}{e^{\alpha q}-1}\left(\int_{0}^{1}e^{\alpha qs}\,\ud s\right)^{\frac{q-1}{q}}
=(e^{\alpha q}-1)^{-1/q}\leq 2^{1/q}e^{-\alpha}\leq 2e^{-\alpha},
\end{align*}
where we have used the assumption $\alpha>\tfrac{1}{q}$ in the penultimate inequality.
\end{proof}

\subsubsection{Acknowledgements}

We would like to thank Hans Zwart and Felix Schwenninger for fruitful discussions and many helpful suggestions. Furthermore, we thank the 
anonymous referee for his careful reading of the first version
and for his  useful comments.

\bibliographystyle{model1b-num-names}

\end{document}